\documentclass{amsart}
\usepackage{amsmath, amsthm}
\usepackage[foot]{amsaddr}
\usepackage{mathtools,hyperref}
\usepackage{amscd,amssymb,amsfonts}
\usepackage[mathscr]{euscript}
\usepackage{enumerate}
\usepackage{stackengine}
\usepackage{tikz}
\usetikzlibrary{shapes}
\usetikzlibrary{decorations.pathreplacing}
\usetikzlibrary{arrows}
\usepackage{multicol}
\usepackage{subcaption}
\usepackage{kbordermatrix}
\usepackage{dsfont}
\DeclarePairedDelimiter\abs{\lvert}{\rvert}%
\DeclareMathOperator{\rank}{rank}

\DeclareMathOperator{\Minors}{Minors}

\DeclareMathOperator{\K}{\\K}
\DeclareMathOperator{\re}{Re}
\DeclareMathOperator{\im}{Im}

\DeclareMathOperator{\CC}{\mathbb{C}}

\DeclareMathOperator{\RR}{\mathbb{R}}

\DeclareMathOperator{\avec}{\mathbf{a}}
\DeclareMathOperator{\bvec}{\mathbf{b}}
\DeclareMathOperator{\cvec}{\mathbf{c}}

\DeclareMathOperator{\vvec}{\mathbf{v}}

\DeclareMathOperator{\x}{\mathbf{x}}

\theoremstyle{plain}
	\newtheorem{theorem}{Theorem}[section]
	
	\newtheorem{proposition}[theorem]{Proposition}
	\newtheorem{corollary}[theorem]{Corollary}

\newtheorem{problem}[theorem]{Problem}
        \theoremstyle{definition}
	\newtheorem{definition}[theorem]{Definition}
	\newtheorem{example}[theorem]{Example}
	\newtheorem*{observation}{Observation}
        \newtheorem*{remark}{Remark}

\begin{document}
\title{Complex psd-minimal polytopes in dimensions two and three}

        \author[Tristram Bogart, Jo\~{a}o Gouveia, and Juan Camilo Torres]{Tristram Bogart$^1$, Jo\~{a}o Gouveia$^2$, and Juan Camilo Torres$^1$}
\date{}
\address{$^1$ Departamento de Matem\'{a}ticas, Universidad de los Andes, Bogot\'{a}, Colombia}
\address{$^2$ University of Coimbra, CMUC, Department of Mathematics, Portugal}
\thanks{Emails: \MakeLowercase{tc.bogart22}@uniandes.edu.co, jgouveia@mat.uc.pt, jc.torresc@uniandes.edu.co}
\thanks{The first and third authors were supported by internal research grants (INV-2020-105-2076 and INV-2018-48-1373, respectively) from the Faculty of Sciences of the Universidad de los Andes. The third author was also supported in his doctoral studies, of which this project forms a part, by the Colombian Government through  Minciencias. The second author was supported by the Centre for Mathematics of the University of Coimbra, grant UIDB/00324/2020, funded by the Portuguese Government through FCT/MCTES.}

         \begin{abstract}
           The extension complexity of a polytope measures its amenability to succinct representations via lifts. There are several versions of extension complexity, including linear, real semidefinite, and complex semidefinite. We focus on the last of these, for which the least is known, and in particular on understanding which polytopes are complex psd-minimal. We prove the existence of an obstruction to complex psd-minimality which is efficiently computable via lattice membership problems. Using this tool, we complete the classification of complex psd-minimal polygons (geometrically as well as combinatorially). In dimension three we exhibit several new examples of complex psd-minimal polytopes and apply our obstruction to rule out many others.
        \end{abstract}

        \maketitle

\section{Introduction}

The extension complexity of a polytope $P$ embedded in $\RR^d$ measures its amenability to succinct representations via lifts. Extension complexity has become popular in recent years because it encodes the complexity  of certain approaches to combinatorial optimization problems and its links to matrix theory have allowed bounds to be proved on them \cite{yannakakis1991expressing,GPT13,fiorini2015exponential,lee-raghavendra-steurer,rothvoss2017}.
Moreover, extension complexity has interesting connections to information theory \cite{braverman2013information,braun2016common,fawzi2015positive}. For instance, extension complexity coincides with the lowest complexity of a random communication protocol that allows two people, independently given a point in the polytope and a valid inequality for the polytope respectively, to jointly evaluate the inequality at that point on average. Different types of extension complexity correspond to different flavours of communication protocols.

There are different definitions of extension complexity depending on what type of lifts are allowed. The most studied version is the \emph{linear extension complexity} which can be defined as the minimum number of facets of a polytope $Q$ that linearly projects onto $P$. Another common version is the \emph{semidefinite extension complexity} which is the smallest number $k$ for which there exists a spectrahedron defined by means of $k \times k$ matrices that linearly projects onto $P$. That is, we ask that $P$ be the linear image of a set of the type
\begin{equation} \label{eq:spectrahedron}
\{x \in \mathbb{R}^n \, : \,  A_0 + A_1 x_1 + \cdots + A_n x_n \succeq 0\}
\end{equation}
where the $A_i$ are real $k \times k$ symmetric matrices.

In this paper we will focus on the \emph{complex semidefinite extension complexity} of a polytope. This is a simple variation of the semidefinite extension complexity in which we again seek to minimize $k$ such that $P$ is the linear image of a set as in (\ref{eq:spectrahedron}), but where we now allow the $A_i$ to be Hermitian matrices rather than only symmetric real matrices. This complex version is, in some ways, more natural for quantum communication complexity (see \cite{LWdW17}).

Our particular focus of interest will be polytopes whose complex semidefinite extension complexity is as small as possible. The minimum possible values of the real and complex semidefinite extension complexities of a $d$-dimensional polytope $P$ are both $d+1$. Accordingly, we call $P$ \emph{psd-minimal} (respectively \emph{complex psd-minimal} if the real (respectively complex) semidefinite extension complexity is equal to $d+1$.
Complex psd-minimal polytopes are an immediate generalization of real psd-minimal polytopes, which are themselves a generalization of $2$-level polytopes (\cite{ACF18}), both of which are interesting classes of low-complexity polytopes. Real psd-minimality was studied in \cite{GRT13,GPRT17} and a full classification was obtained up to dimension four. There are two combinatorial classes of psd-minimal polytopes in dimension two (triangles and quadrilaterals), six in dimension three and $32$ in dimension four. In higher dimension very little is known, except that any $d$-polytope with at most $d+2$ vertices or facets is psd-minimal. The complex case is even more open, but the only existing results, from \cite{GGS17}, hint at significant differences from the real case: while no pentagon in $\RR^2$ is complex psd-minimal, the regular hexagon is known to be so.

In this paper, we introduce a new tool based on lattice membership problems to identify obstructions to complex psd-minimality. Using this tool, we characterize which hexagons are complex psd-minimal and show that these hexagons, together with triangles and quadrilaterals, are the only complex psd-minimal polygons. We then proceed to make inroads on the problem of classifying complex psd-minimal 3-polytopes. In particular, our precise classification of complex psd-minimal hexagons allows us to reduce this problem to the study of the class of \emph{doubly 3/4 polytopes}: those for which each vertex and each facet has degree at most four. We also construct a series of new examples of complex psd-minimal 3-polytopes.

\section{Complex psd-minimality through slack matrices} \label{sec:background}
One of the reasons to focus on complex psd minimal polytopes is that while in general checking the real or complex semidefinite extension complexity of a polytope is very hard, minimality can be characterized by a simple algebraic condition. Extension complexity is intimately connected to the properties of a special matrix associated to a polytope: its slack matrix.

\begin{definition} 
  Let $P\subseteq\mathbb{R}^d$ be a full-dimensional polytope with vertices $\mathbf{v}_1,\ldots,\mathbf{v}_n$ and facets $F_1,\ldots,F_m$. Then $P=\lbrace \mathbf{t}\in\mathbb{R}^d: A\mathbf{t}+\mathbf{b}\geq\mathbf{0}\rbrace$ for some $A=[a_{ij}]_{m\times d}\in\mathbb{R}^{m\times d}$, $\mathbf{b}\in\mathbb{R}^m$, and such that $F_i=\lbrace \mathbf{t}\in P:a_{i1}t_1+\cdots+a_{id}t_d +b_i=0\rbrace$ for all $i=1,\ldots,m$. If $h_i(t):=a_{i1}t_1+\cdots+a_{id}t_d +b_i$ for $i=1,\ldots,m$, then the matrix defined as $[h_i(\mathbf{v}_j)]_{m\times n}$ is called a \emph{slack matrix} of $P$. If we take a slack matrix of $P$ and replace each non-zero entry with a distinct variable, we obtain the \emph{symbolic slack matrix} of $P$.
\end{definition}

Extension complexities of a polytope $P$ are all equivalent to certain factorization ranks of its slack matrix (see \cite{yannakakis1991expressing,GPT13,fawzi2015positive}). These are still notoriously hard quantities to compute, but in the case of minimal complex psd-minimality one has a more concrete criterion.

\begin{proposition}
A $d$-polytope $P$ with slack matrix $S$ is complex psd-minimal if and only if one can find a rank $d+1$ complex matrix $A$ such that each entry of $S$ is the square of the absolute value of the corresponding entry of $A$.
\end{proposition}

This criterion can be found in \cite{GGS17} and is an immediate adaptation of its real analogue, found in \cite{GPT13}. Note that verifying it remains a notoriously hard problem, since it corresponds to verifying membership on a certain amoeba of a rank variety. For a more general take on this problem, see \cite{GGphaseless}. However, the algebraic nature of this characterization will allow us to explicitly construct obstructions to complex psd-minimality in certain cases. To do that we will start by encoding the possible realizations of any polytope $P$ in a more algebraic way.

Slack matrices are a very useful way to encode polytopes. A deeper look into these matrices can be found in \cite{gouveia2013nonnegative, GMTW18} but we will now state the most important properties that we will use in this paper.
The first important property is that slack matrices, up to column and row scaling by positive scalars,  characterize polytopes up to projective equivalence.
\begin{proposition}
Given two polytopes $P$ and $Q$ with slack matrices $S_P$ and $S_Q$, there exist diagonal matrices $D_1$ and $D_2$ with positive diagonal entries such that $S_P=D_1 S_Q D_2$ if and only if there exists a projective transformation that sends $P$ to $Q$.
\end{proposition}

The second important property is a characterization of slack matrices. Given an abstract $d$-dimensional polytope $P$, we want to consider all of its possible realizations. If we are interested in the realizations only up to projective equivalence, a natural idea is to study the set of \emph{scaled slack matrices} of $P$, which we define as the set of matrices that can be obtained by scaling rows and columns of the slack matrix of any realization of $P$ by positive scalars. These matrices have a very simple characterization.

\begin{proposition}\label{prop:rank}
 If $S_P(x_1,\dots,x_m)$ is the symbolic slack matrix of a $d$-polytope $P$ and $\boldsymbol{\zeta}\in(\CC^\ast)^m$, then $\rank S_P(\boldsymbol{\zeta})\geq d+1$. Furthermore, a matrix $S$ is a scaled slack matrix of $P$ if and only if $\rank(S)=d+1$ and there exists $\boldsymbol{\alpha} \in \RR_{++}^m$ such that $S=S_P(\boldsymbol{\alpha})$.
\end{proposition}

This immediately suggests that we consider the ideal generated by all $(d+2)$-minors of the symbolic slack matrix. By the above discussion, the positive points in its variety are in one to one correspondence with scaled slack matrices of $P$. We will call this ideal the \emph{minor ideal} of $P$. We can enlarge the ideal further by saturation with respect to all of the variables. This corresponds to considering the ideal of all polynomials which after multiplication by some monomial belong to the minor ideal. We call this the \emph{slack ideal} of $P$, and its variety has the same positive part as the minor ideal. We can and will consider the slack ideal as extended to the Laurent polynomial ring.

\begin{observation}
  As is noted in \cite{GMTW18}, we can scale the rows and columns of the symbolic slack matrix of a polytope $P$ so that the entries indexed by the edges of a maximal spanning forest of the nonincidence graph $\mathbf{G}_P$ of vertices and facets are all equal to 1. The rank restriction of the slack matrix may fix some of the remaining entries, as in the following example which we will use later on.
  \end{observation}

\begin{example} \label{ex:squareslack}
The symbolic slack matrix of any quadrilateral $Q$ is of the form
\[M_Q = \kbordermatrix{
& \mathbf{v}_1 & \mathbf{v}_2 & \mathbf{v}_3 & \mathbf{v}_4\\
F_1 & 0 & 0 & x_{13} & x_{14}\\
F_2 & x_{21} & 0 & 0 & x_{24} \\
F_3 & x_{31} & x_{32} & 0 & 0\\
F_4 & 0 & x_{42} & x_{43} & 0},\]
which can be scaled to the matrix
\[ M_Q' = \kbordermatrix{
& \mathbf{v}_1 & \mathbf{v}_2 & \mathbf{v}_3 & \mathbf{v}_4\\
F_1 & 0 & 0 & 1 & x\\
F_2 & 1 & 0 & 0 & 1\\
F_3 & 1 & 1 & 0 & 0\\
F_4 & 0 & 1 & 1 & 0}.\]
Notice that there is a 1 in each of the entries indexed by the edges in a maximal spanning forest of $\mathbf{G}_Q$. Since the rank of a slack matrix of $Q$ is 3, and $x-1$ is a 4-minor (i.e. the determinant) of $M_Q'$, its entry in position $(F_1,\mathbf{v}_4)$ must also be a 1.
\end{example}

Our main tool for showing that a polytope cannot be complex psd-minimal is the following result.
\begin{proposition}\label{prop:to}
Let $P$ be a $d$-polytope, and let $S_P(\x)=S_P(x_1,\ldots,x_k)$ be a scaled symbolic slack matrix of $P$. Suppose there is a trinomial of the form $\x^\mathbf{a}-\x^\mathbf{b}+\x^\mathbf{c}$, $\mathbf{a},\mathbf{b},\mathbf{c}\in\mathbb{N}^k$, in $\Minors_{d+2}\left(S_P(\mathbf{x})\right)$. If there are $\boldsymbol{\alpha},\boldsymbol{\zeta}\in\mathbb{C}^k$ such that $S_P(\boldsymbol{\alpha})=S_P(\boldsymbol{\zeta})\odot\overline{S_P(\boldsymbol{\zeta})}$ with $\rank S_P(\boldsymbol{\alpha})=\rank S_P(\boldsymbol{\zeta})=d+1$, then $\re(\boldsymbol{\zeta}^\mathbf{a}\overline{\boldsymbol{\zeta}^\mathbf{c}})=0$.
\end{proposition}

\begin{proof}
Since $\x^\mathbf{a}-\x^\mathbf{b}+\x^\mathbf{c}\in\Minors_{d+2}\left(S_P(\mathbf{x})\right)$ and $\rank S_P(\boldsymbol{\alpha})=\rank S_P(\boldsymbol{\zeta})=d+1$, we have $\boldsymbol{\alpha}^\mathbf{a}-\boldsymbol{\alpha}^\mathbf{b}+\boldsymbol{\alpha}^\mathbf{c}=0$ and $\boldsymbol{\zeta}^\mathbf{a}-\boldsymbol{\zeta}^\mathbf{b}+\boldsymbol{\zeta}^\mathbf{c}=0$. So $\boldsymbol{\alpha}^\mathbf{b}=\boldsymbol{\alpha}^\mathbf{a}+\boldsymbol{\alpha}^\mathbf{c}$ and $\boldsymbol{\zeta}^\mathbf{b}=\boldsymbol{\zeta}^\mathbf{a}+\boldsymbol{\zeta}^\mathbf{c}$. Thus
\[ \boldsymbol{\alpha}^\mathbf{a}+\boldsymbol{\alpha}^\mathbf{c}
= \boldsymbol{\alpha}^\mathbf{b} =\boldsymbol{\zeta}^\mathbf{b}\overline{\boldsymbol{\zeta}^\mathbf{b}} = \left(\boldsymbol{\zeta}^\mathbf{a}+\boldsymbol{\zeta}^\mathbf{c}\right)\left(\overline{\boldsymbol{\zeta}^\mathbf{a}}+\overline{\boldsymbol{\zeta}^\mathbf{c}}\right) =\boldsymbol{\alpha}^\mathbf{a}+2\re(\boldsymbol{\zeta}^\mathbf{a}\overline{\boldsymbol{\zeta}^\mathbf{c}})+\boldsymbol{\alpha}^\mathbf{c}. \]

It follows that $\re(\boldsymbol{\zeta}^\mathbf{a}\overline{\boldsymbol{\zeta}^\mathbf{c}})=0$.
\end{proof}

We can assume by monomial scaling that each trinomial takes the form $\x^\mathbf{a} - \x^\mathbf{b} + 1$, where $\x^\mathbf{a}$ and $\x^\mathbf{b}$ are Laurent monomials, and simplify the trinomial obstruction as follows.

\begin{corollary}\label{cor:to}
Suppose the slack ideal contains a trinomial $\x^\mathbf{a} - \x^\mathbf{b} + 1$. If there are $\boldsymbol{\alpha},\boldsymbol{\zeta}\in\mathbb{C}^k$ such that $S_P(\boldsymbol{\alpha})=S_P(\boldsymbol{\zeta})\odot\overline{S_P(\boldsymbol{\zeta})}$ with $\rank S_P(\boldsymbol{\alpha})=\rank S_P(\boldsymbol{\zeta})=d+1$, then $\re(\boldsymbol{\zeta}^\mathbf{a}) = 0$.
\end{corollary}

This form of the obstruction allows us to easily search for incompatible sets of trinomials and binomials: it boils down to a lattice membership problem which can be efficiently solved via Hermite normal form.

\begin{proposition}\label{prop:incompatible}
  Suppose the slack ideal contains trinomials
  \[ \x^{\avec_1} - \x^{\bvec_1} + 1, \dots, \x^{\avec_k} - \x^{\bvec_k} + 1 \]
  and binomials
  \[ \x^{\cvec_1} - 1, \dots, \x^{\cvec_s} - 1.\]
  If there exists $j$ such that either:
  \begin{itemize} \item $\bvec_j$ belongs to the lattice $L$ generated by $\{ \avec_1, \dots, \avec_k, \cvec_1, \dots, \cvec_s \}$, or
    \item $2 \bvec_j$ belongs to the lattice $L'$ generated by $\{ \avec_i+\avec_t:1\leq i,t\leq k\}\cup\{\cvec_1, \dots, \cvec_s \}$,
  \end{itemize}
  then there are no $\boldsymbol{\alpha},\boldsymbol{\zeta}\in \left(\mathbb{C}^\ast\right)^k$ such that $S_P(\boldsymbol{\alpha})=S_P(\boldsymbol{\zeta})\odot\overline{S_P(\boldsymbol{\zeta})}$ with $\rank S_P(\boldsymbol{\alpha})=\rank S_P(\boldsymbol{\zeta})=d+1$, and thus $P$ is not complex psd-minimal.
  \end{proposition}


\begin{proof}
  By Corollary \ref{cor:to}, if there exist such $\boldsymbol{\alpha}$ and $\boldsymbol{\zeta}$, then $\boldsymbol{\zeta}^{\avec_i}$ is pure imaginary for all $i$. It also follows directly from each binomial equation $\x^{\cvec_\ell} = 1$ that $\boldsymbol{\zeta}^{\cvec_\ell}$ equals 1; in particular it is real. If $\bvec_j$ belongs to $L$,
  then $\boldsymbol{\zeta}^{\bvec_j}$ is a product of pure imaginary and pure real numbers which is again either real or imaginary. However the equation $\x^{\avec_j} + 1 = \x^{\bvec_j}$ cannot be satisfied by any $\boldsymbol{\zeta}$ such that $\boldsymbol{\zeta}^{\avec_j}$ is a nonzero imaginary number and $\boldsymbol{\zeta}^{\bvec_j}$ either real or imaginary.

  Similarly, if $2 \bvec_j$ is a sum of generators of $L'$ then $\boldsymbol{\zeta}^{2 \bvec_j}$ is real and thus $\boldsymbol{\zeta}^{\bvec_j}$ is either real or imaginary and we reach the same conclusion.
\end{proof}

\begin{remark}
  \begin{enumerate}[1.] $ $
  \item If the slack ideal contains a polynomial whose terms all have the same sign, then there are no strictly positive points in the slack variety at all; that is, the polytope is not even realizable. In fact, finding polynomials (and trinomials in particular) with every sign positive is a classical way of proving non-realizability of spheres (see \cite{BOKOWSKI199021}). Since in this paper we are dealing only with polytopes, this rules out binomials and trinomials of sign patterns other than those considered in Proposition~\ref{prop:incompatible}.

  \item It might happen (in theory) that an appropriately-sized minor of the slack matrix (or a polynomial in the ideal generated by such minors) is a binomial or trinomial with coefficients other than $\pm 1$; this would require that two or more terms of the minor happened to equal the same monomial. However the coefficients would still be real numbers (in fact integers) and so the proof of Proposition \ref{prop:incompatible} would still go through.
\end{enumerate}
\end{remark}

\textbf{Algorithms, additional computational calculations and data.} We can design an algorithm using Proposition \ref{prop:incompatible} and the Hermite normal form to rule out non-complex psd-minimal polytopes.

Let us recall the Hermite normal form of an integer matrix which can be used to test membership in a lattice: for any matrix $A\in\mathbb{Z}^{m\times n}$, there are unique matrices $H=H(A)\in\mathbb{Z}^{m\times n}$ and $U=U(A)\in\mathbb{Z}^{m\times m}$ such that
\begin{enumerate}[1.]
\item $H = UA$
\item $U$ is unimodular
\item $H$ is in row echelon form, each pivot is positive, and the elements above it are nonnegative and smaller than the pivot.
\end{enumerate}
If $B$ is obtained from $A$ by adding a new row $b$ at the end, then $b$ belongs to the lattice generated by the rows of $A$ if and only if $H(B)$ is equal to $H(A)$ with one extra row of zeros at the end. Let us see how we can explicitly write $b$ as an integer combination of rows of $A$ if this is the case. We have that $H(A)=U(A)A$ and $H(B)=U(B)B$. If $\hat{I}$ is the $m\times m$ identity matrix with one additional row of zeros at the end, then $H(B)=\hat{I}H(A)$, and thus $B=U(B)^{-1}\hat{I}U(A)A$. So if $[c_1,\ldots,c_m]$ is the last row of $U(B)^{-1}\hat{I}U(A)$, then $b=c_1(\text{first row of $A$})+\cdots+c_m  (\text{last row of $A$})$.

The criteria in Proposition \ref{prop:incompatible} require a list of binomials and trinomials in the slack ideal. We can thus work, for example, with the binomials and trinomials that are $(d+2)$-minors of the symbolic slack matrix. With SageMath \cite{sagemath} we implemented two functions which determine whether it can be deduced that the $d$-polytope $P$ is not complex psd-minimal using the lattice $L$ or $L'$, respectively, and the binomials and trinomials that are $(d+2)$-minors of the symbolic slack matrix. 

This SageMath Notebook can be found at \href{https://sites.google.com/view/jctorres}{https://sites.google.com/view/jctorres} under the name \texttt{ComplexPsdMinimality.ipynb}, the functions mentioned above are named ComplexPsdMinimality1 and ComplexPsdMinimality2 respectively; an example is included in the file. Other algorithms, additional computational calculations and data to which we refer in this paper can be found on the webpage.

\section{Polygons} \label{sec:2D}
We now apply our method to classify the complex psd-minimal polygons. To do this, we first identify the trinomial 4-minors of the symbolic slack matrix of a polygon.

\begin{proposition} \label{prop:2Dtrinomials}
  Let $P$ be an $n$-gon for $n \geq 5$ with vertices $\vvec_1, \dots, \vvec_n$ and facets $F_1, \dots F_n$ where (taken mod $n$) $\vvec_i$ and $\vvec_{i+1}$ are incident to $F_i$. Then exactly $2n$ of the 4-minors of $S_P(\x)$ are trinomials and these trinomials are of the form
  \begin{eqnarray*} f_i = & x_{i, i+2}x_{i+1,i+3}x_{i+2,i+1} - x_{i,i+2}x_{i+1,i+3}x_{i+2,i}x_{i+3,i+1} + x_{i,i+3}x_{i+1,i}x_{i+2,i+1}; \\
    g_i = & x_{i,i+2}x_{i+1,i+3}x_{i+2,i+4}x_{i+3,i+1} - x_{i,i+4}x_{i+1,i+3}x_{i+2,i+1}x_{i+3,i+2} \\
    & +  x_{i,i+3}x_{i+1,i+4}x_{i+2,i+1}x_{i+3, i+2}
    \end{eqnarray*}
  for each $i=1,\dots,n$.
  \end{proposition}

\begin{proof}
  The trinomial $f_i$ is the minor of $S_P(\x)$ obtained from a path \\
  $\vvec_i F_i \vvec_{i+1} F_{i+1} \vvec_{i+2} F_{i+2} \vvec_{i+3} F_{i+3}$ on the boundary of $P$, which is (assuming $i=1$ for notational conveniency) the determinant of the matrix
\[\kbordermatrix{
& \mathbf{v}_1 & \mathbf{v}_2 & \mathbf{v}_3 & \mathbf{v}_4\\
F_1 & 0 & 0 & x_{13} & x_{14 }\\
F_2 & x_{21} & 0 & 0 & x_{24}\\
F_3 & x_{31} & x_{32} & 0 & 0\\
F_4 & x_{41} & x_{42} & x_{43} & 0}.\]


Similarly, $g_i$ is the minor obtained from a path that starts with a facet; that is, the path from $F_i$ to $\vvec_{i+4}$.

The proof that there are no other trinomial 4-minors is by case analysis, which we omit. The number of cases is limited because $S_P(\x)$ has no $2\times 2$ submatrices of zeros and only two zeros in each row and in each column.

\end{proof}

\begin{proposition}
Pentagons are not complex psd-minimal.
\end{proposition}

\begin{proof}
  By monomial scalings of the trinomial minors $f_1$, $f_2$, and $g_1$ in the slack ideal of a pentagon, we obtain
  \begin{eqnarray*}
    \tilde{f_1} & = & \frac{x_{14}x_{21}x_{43}}{x_{13}x_{24}x_{41}} - \frac{x_{31}x_{42}}{x_{32}x_{41}} + 1 \\
    \tilde{f_2} & = & \frac{x_{25}x_{32}x_{54}}{x_{24}x_{35}x_{52}} - \frac{x_{42}x_{53}}{x_{43}x_{52}} + 1 \\
    \tilde{g_1} & = & \frac{x_{31}x_{43}x_{15}}{x_{35}x_{41}x_{13}} - \frac{x_{53}x_{14}}{x_{54}x_{13}} + 1 \end{eqnarray*}

  The product of the initial positive terms of $\tilde{f_1}$ and $\tilde{f_2}$ equals the negative term of $\tilde{g_1}$, so the result follows from Proposition \ref{prop:incompatible}.

\end{proof}

\begin{proposition}\label{prop:ngons}
For $n \geq 7$, $n$-gons are not complex psd-minimal.
\end{proposition}

\begin{proof}
Let $P$ be an $n$-gon for $n \geq 7$.  It is not hard to see from the classification of trinomial 4-minors in Proposition \ref{prop:2Dtrinomials} that these alone do not suffice to produce an obstruction to complex psd-minimality. However, the submatrix of $S_P(\x)$ obtained from a path of six vertices and six facets is
\[ M = \begin{bmatrix}
    0 & 0 & x_{13} & x_{14} & x_{15} & x_{16} \\
    x_{21} & 0 & 0 & x_{24} & x_{25} & x_{26} \\
    x_{31} & x_{32} & 0 & 0 & x_{35} & x_{36} \\
    x_{41} & x_{42} & x_{43} & 0 & 0 & x_{46} \\
    x_{51} & x_{52} & x_{53} & x_{54} & 0 & 0 \\
    x_{61} & x_{62} & x_{63} & x_{64} & x_{65} & 0
  \end{bmatrix} \]
  and an appropriate scaling of its rows and columns yields
\[ M' = \begin{bmatrix}
    0 & 0 & 1 & x_{14} & x_{15} & 1 \\
    1 & 0 & 0 & 1 & x_{25} & x_{26} \\
    x_{31} & 1 & 0 & 0 & 1 & x_{36} \\
    1 & x_{42} & 1 & 0 & 0 & x_{46} \\
    x_{51} & 1 & x_{53} & 1 & 0 & 0 \\
    x_{61} & x_{62} & x_{63} & x_{64} & 1 & 0
  \end{bmatrix} .\]

 Even after this scaling, the trinomial 4-minors of $M'$ are not sufficient to directly produce an obstruction. Consider instead the following trinomials and their canonical Laurent forms (as in Proposition \ref{prop:incompatible}):
\[ \begin{array}{l|l}
    x_{25} - x_{42}x_{53} + 1 & x_{25}-x_{42}x_{53} + 1 \\
  x_{46} - x_{15}x_{26} + x_{25} & x_{46}/x_{25} - x_{15}x_{26}/x_{25} + 1 \\
   x_{42}x_{51}-1+x_{15} & x_{42}x_{51} / x_{15} - 1 / x_{15}+1 \\
    x_{15} - x_{14}x_{25}+x_{42} & x_{15}/ x_{42} - x_{14}x_{25} / x_{42} + 1 \\
    x_{26}-x_{25}x_{36}+x_{46}x_{53} & x_{26}/ x_{46}x_{53} - x_{25}x_{36}/ x_{46}x_{53} + 1 \\
   x_{46}x_{53}x_{62}-x_{46}x_{63}+x_{26} & x_{46}x_{53}x_{62}/x_{26}-x_{46}x_{63}/x_{26} + 1 \\
   x_{14}-x_{31}x_{42}+1 & x_{14}-x_{31}x_{42}+1 \\
   x_{46}x_{61} - x_{46}x_{51}x_{62} + 1 & x_{46}x_{61} - x_{46}x_{51}x_{62} + 1.
 \end{array} \]
All the trinomials listed belong to the ideal generated by all 4-minors\footnote{See \texttt{AdditionalCalculations1.ipynb} in \href{https://sites.google.com/view/jctorres}{https://sites.google.com/view/jctorres}}.
The product of the initial positive terms of the first six Laurent trinomials is $x_{46}x_{51}x_{62}$ which is exactly the negative term of the last Laurent trinomial, so the result follows from Proposition \ref{prop:incompatible}. Note that the seventh trinomial was not necessary here, but will be used in a later proof.
\end{proof}

It remains to consider the case of hexagons. In this case we will go beyond the usual consideration of combinatorial type, and identify the precise class of hexagons embedded in $\RR^2$ that are complex psd-minimal. This result will be needed in the following section in order to show that large classes of (combinatorial) 3-polytopes are not complex psd-minimal.

\begin{definition}[Pappus Hexagon]
A hexagon with consecutive vertices $\mathbf{v}_1,\ldots$, $\mathbf{v}_6$ is called a \emph{Pappus hexagon} if the following conditions hold:
\begin{itemize}
\item The lines $\overleftrightarrow{\mathbf{v}_1\mathbf{v}_2}$, $\overleftrightarrow{\mathbf{v}_3\mathbf{v}_6}$ and $\overleftrightarrow{\mathbf{v}_4\mathbf{v}_5}$ are concurrent.
\item The lines $\overleftrightarrow{\mathbf{v}_2\mathbf{v}_3}$, $\overleftrightarrow{\mathbf{v}_1\mathbf{v}_4}$ and $\overleftrightarrow{\mathbf{v}_5\mathbf{v}_6}$ are concurrent.
\item The lines $\overleftrightarrow{\mathbf{v}_1\mathbf{v}_6}$, $\overleftrightarrow{\mathbf{v}_2\mathbf{v}_5}$ and $\overleftrightarrow{\mathbf{v}_3\mathbf{v}_4}$ are concurrent.
\end{itemize}
\end{definition}

Here \textit{concurrent} means that the lines intersect at a single point or are parallel (intersect at infinity). We use this name because the dual version of Pappus' Theorem in projective geometry states that if a hexagon satisfies two of the three conditions, then it also satisfies the third. There is a related notion of Desarguian hexagon which plays a role in the theory of linear extension complexity \cite{P16}.

\begin{proposition} \label{prop:Pappus} A hexagon is complex psd-minimal if and only if it is a Pappus hexagon.
\end{proposition}

\begin{proof}
  Let $N$ be the slack matrix of a hexagon. Then $N$ is equal to the matrix $M$ obtained above from a path of vertices and facets in an $n$-gon for $n \geq 7$ except that the variable $x_{61}$ is replaced by zero, since the last facet in the path is now incident to the first vertex.  Let $N'$ be obtained from $N$ by applying the same row and column scaling used to obtain $M'$ from $M$. Then the ideal $I$ of 4-minors of $N'$ contains all of the same trinomials that appeared in the proof of Proposition \ref{prop:ngons} except for the last one that involves $x_{61}$.

  Suppose the hexagon is complex psd-minimal, and $\boldsymbol{\alpha}$ and $\boldsymbol{\zeta}$ belong to the variety of $I$ with $\boldsymbol{\zeta}$ being a coordinate-wise square root of $\boldsymbol{\alpha}$ as above. By combining the third and fourth trinomials in the list, we see that
  \begin{equation}\zeta_{42} - \zeta_{14}\zeta_{25} = \zeta_{42}\zeta_{51}-1. \label{eq:Pappus} \end{equation}
  By applying Corollary \ref{cor:to} to the same two trinomials, we obtain that $\zeta_{42}\zeta_{51}/\zeta_{15}$ and $\zeta_{15}/\zeta_{42}$ are imaginary, so their product $\zeta_{51}$ is real. By applying the same corollary to the first and seventh trinomials in the list, we also obtain that $\zeta_{25}$ and $\zeta_{14}$ are imaginary, so their product $\zeta_{14}\zeta_{25}$ is real. Thus we can take the imaginary part of (\ref{eq:Pappus}) to obtain
  \[ \im(\zeta_{42}) = \zeta_{51}\im(\zeta_{42}). \]
  Reader can easily check that if $\zeta_{42}$ is real, then $\zeta_{51}$ must be equal to zero, which is a contradiction. Thus $\im(\zeta_{42})\neq 0$, and so $\zeta_{51} = 1$.

  Now $\zeta_{51} = 1$ implies that $\alpha_{51} = 1$ as well. By the chosen scaling, we also have $\alpha_{54} = \alpha_{21} = \alpha_{24} =1$. Since $S_P(\boldsymbol{\alpha})$ is a scaled slack matrix of $P$, there is some rescaling of the columns that makes it a true slack matrix, and by rescaling rows we can think of entry $(i,j)$ as the distance $\textup{dist}_{l_i}(\vvec_j)$ between vertex $\vvec_j$ and the line $l_i$ that is the affine span of facet $F_i$. Since scaling rows and columns preserves minors we have the following ratio equality
  $$ \frac{\textup{dist}_{l_5}(\vvec_1)}{\textup{dist}_{l_5}(\vvec_4)}=\frac{\textup{dist}_{l_2}(\vvec_1)}{\textup{dist}_{l_2}(\vvec_4)}.$$
  By a similarity of triangles argument this can be seen to be equivalent to $l_2, l_5$ and the line spanned by the vertices $\vvec_1$ and $\vvec_4$ being either parallel (if the ratios are one) or concurrent in a point.
%
%

%

  Conversely, let $H$ be a Pappus hexagon with vertices consecutively labeled $\vvec_1$, $\vvec_2$, $\dots$, $\vvec_6$ as above. Up to projective transformation, we may assume that
  \[ \vvec_1 = (0,0), \; \vvec_2 = (1,0), \; \vvec_5 = (1,1), \; \vvec_6 = (0,1).\]
 By concurrency the line $\overleftrightarrow{\vvec_3 \vvec_4}$ is vertical; let this line be $x=1+b$. Again by concurrency, let $(1+a,1)$ be the common intersection point of $\overleftrightarrow{\vvec_2\vvec_3}$, $\overleftrightarrow{\vvec_1\vvec_4}$ and $\overleftrightarrow{\vvec_5\vvec_6}$. Then
 \[ \vvec_3 = (1+b, b/a), \; \vvec_4 = \left( 1+b, (1+b)/(1+a) \right) \]
 so $H$ is as shown in Figure \ref{fig:Pappus}

 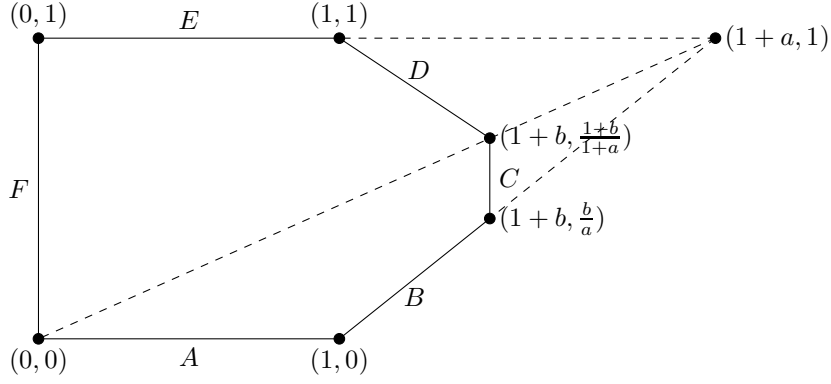
\begin{figure}[ht]
  \begin{center}
 \begin{tikzpicture}
\filldraw (0,0) node[below] {$(0,0)$} circle (2pt);
\filldraw (4,0) node[below] {$(1,0)$} circle (2pt);
\filldraw (6,1.6) node[right] {$(1+b,\frac{b}{a})$} circle (2pt);
\filldraw (6,2.67) node[right] {$(1+b,\frac{1+b}{1+a})$} circle (2pt);
\filldraw (4,4) node[above] {$(1,1)$} circle (2pt);
\filldraw (0,4) node[above] {$(0,1)$} circle (2pt);
\filldraw (9,4) node[right] {$(1+a,1)$} circle (2pt);
\draw[dashed] (9,4) -- (4,4);
\draw[dashed] (9,4) -- (0,0);
\draw[dashed] (9,4) -- (6,1.6);
\draw (0,0) -- (4,0) node[below, midway] {$A$};
\draw (4,0) -- (6,1.6) node[below, midway] {$B$};
\draw (6,1.6) -- (6,2.67) node[right, midway] {$C$};
\draw (6,2.67) -- (4,4) node[above, midway] {$\ D$};
\draw (4,4) -- (0,4) node[above, midway] {$E$};
\draw (0,0) -- (0,4) node[left, midway] {$F$};
 \end{tikzpicture}

  \end{center}
  \caption{A Pappus hexagon up to projective transformation} \label{fig:Pappus}
 \end{figure}

 and a scaled slack matrix of $H$ is\footnote{See \texttt{AdditionalCalculations2.ipynb} in \href{https://sites.google.com/view/jctorres}{https://sites.google.com/view/jctorres}}
 \[S_{a,b}:=\begin{bmatrix}
0 & 0 & 1 & \frac{1}{a} & \frac{a-b}{a(1+b)} & 1\\
1 & 0 & 0 & 1 & a & \frac{a(1+a)(1+b)}{a-b}\\
\frac{1+b}{b} & 1 & 0 & 0 & 1 & \frac{a(1+b)^2}{b(a-b)}\\
1 & \frac{(1+a)b}{a(1+b)} & 1 & 0 & 0 & 1\\
1 & 1 & \frac{a(1+b)}{b} & 1 & 0 & 0\\
0 & 1 & \frac{a^2(1+b)^2}{b(a-b)} & \frac{(1+a)(1+b)}{a-b} & 1 & 0
 \end{bmatrix}.\]
Now let \[M_{\xi_1,\xi_2}:=\begin{bmatrix}
0 & 0 & 1 & \frac{1}{\xi_1} & \xi_2 & 1\\
1 & 0 & 0 & 1 & \xi_1 & \frac{1+\xi_1}{\xi_2}\\
\frac{1+\xi_1}{\xi_1(1-\xi_2)} & 1 & 0 & 0 & 1 & \frac{1+\xi_1}{\xi_1\xi_2(1-\xi_2)}\\
1 & 1-\xi_2 & 1 & 0 & 0 & 1\\
1 & 1 & \frac{1+\xi_1}{1-\xi_2} & 1 & 0 & 0\\
0 & 1 & \frac{1+\xi_1}{\xi_2(1-\xi_2)} & \frac{1+\xi_1}{\xi_1\xi_2} & 1 & 0
\end{bmatrix}\]
where $\xi_1=\pm\sqrt{a}i$ and $\xi_2=\frac{a-b}{a(1+b)}\pm\frac{\sqrt{(1+a)b(a-b)}}{a(1+b)}i$ (each of the four possibilities work). We see that $S_{a,b}=M_{\xi_1,\xi_2}\odot\overline{M_{\xi_1,\xi_2}}$. The rank of $M_{\xi_1,\xi_2}$ is at least three since it has the zero pattern of the slack matrix of a hexagon, and we can check computationally that all of its 4-minors vanish\footnote{See \texttt{AdditionalCalculations2.ipynb} in \href{https://sites.google.com/view/jctorres}{https://sites.google.com/view/jctorres}}. Thus its rank is exactly three, so $H$ is complex psd-minimal.
\end{proof}

This implies that we can represent any Pappus hexagon as the projection of a complex $3 \times 3$ spectrahedral lift. This lift can in fact be derived explicitly from the matrix $S_P(\boldsymbol{\zeta})$ that we constructed.

\begin{example}
The regular hexagon $H$ of vertices $\{(\cos(\frac{k\pi}{3}),\sin(\frac{k\pi}{3})), k=0,...,5\}$ can be written as the set if all $(x_1,x_2) \in \mathbb{R}^2$ such that we can find real numbers $x_3, x_4, x_5$ for which
$$
\begin{bmatrix}
 2-\frac{4}{\sqrt{3}}{x_2} & 1+{x_3}+i {x_4}-{x_5} & 1+{x_3}+i {x_4}+{x_5} \\
1+ {x_3}-i {x_4}-{x_5} & 1+{x_1}+\sqrt{3} {x_2}+2 {x_3}-2 {x_5} & 1+{x_3}-i
   {x_5} \\
1+ {x_3}-i {x_4}+{x_5} & 1+{x_3}+i {x_5} &1 -{x_1}-\frac{1}{\sqrt{3}}{x_2} \\
\end{bmatrix}
 \succeq 0.$$
A step by step derivation of a $3 \times 3$ lift from the $S_P(\boldsymbol{\zeta})$ above, from which this Example is a simple variation, can be found in Chapter 4.2 of \cite{torres2020slack}.
\end{example}

\section{3-polytopes} \label{sec:3D}
The obstructions we described in Section \ref{sec:background} can be applied to polytopes in any dimension. The following result, whose proof is virtually the same as in the real case (see \cite[Proposition 3.8]{GRT13} limits the search for complex psd-minimal polytopes in dimensions three and beyond.

\begin{proposition}\label{prop:faces}
If $P$ is a complex psd-minimal polytope, then all of its faces are also complex psd-minimal.
\end{proposition}

We will focus on the three-dimensional case, for which we have the following specialization.

\begin{proposition} \label{prop:346}
If $P$ is a complex psd-minimal 3-polytope, then all of its facets are triangles, quadrilaterals or hexagons, and all of its vertices have degree three, four or six.
\end{proposition}
\begin{proof}
The first statement follows immediately from Proposition \ref{prop:faces} and the classification of combinatorial complex psd-minimal polygons in the previous section. The second then follows from the first and the fact that complex psd-minimality is preserved under duality.
\end{proof}

To apply our method systematically, we must first identify combinatorial configurations that yield binomials or trinomials in the slack ideal, as we did in Proposition \ref{prop:2Dtrinomials} for the two-dimensional case.

\begin{proposition} \label{prop:3d-trinomials}
  Let $P$ be a 3-polytope and suppose that there exists a quadrilateral facet $F_4$ incident to a vertex $\vvec_4$ of degree four. Let $\vvec_1$, $\vvec_2$, and $\vvec_3$ be the other vertices incident to $F_4$ and $F_1$, $F_2$, and $F_3$ be the other facets incident to $\vvec_4$, arranged as in Figure \ref{fig:3d-minors} (A).  Then
  \[ g = yz \left( x_{12}x_{23}x_{31} - x_{13}x_{22}x_{31} + x_{13}x_{21}x_{32} \right) \]
  is a trinomial 5-minor of $S_P(\x)$ for some variables $y, z$.

  Thus $x_{12}x_{23}x_{31} - x_{13}x_{22}x_{31} + x_{13}x_{21}x_{32}$ belongs to the slack ideal of $P$.
\end{proposition}

\begin{proof}
  The submatrix of $S_P(\x)$ with rows indexed by $F_1, \dots, F_4$ and any other single facet and columns indexed by $\vvec_1, \dots, \vvec_4$ and any other single vertex is of the form
  $ \begin{bmatrix}
    0 & x_{12} & x_{13} & 0 & \ast \\
    x_{21} & x_{22} & x_{23} & 0 & \ast \\
    x_{31} & x_{32} & 0 & 0 & \ast \\
    0 & 0 & 0 & 0 & y \\
    \ast & \ast & \ast & z & \ast
  \end{bmatrix} $
  (where $\ast$ may denote either zero or a variable) and its determinant is exactly $g$.
\end{proof}
\begin{figure}[ht]
  \begin{subfigure}[c]{.45\linewidth}
    \begin{center}
    \begin{tikzpicture}[scale=0.4]
\filldraw (0,0) node[below left] {$\mathbf{v}_4$} circle (2pt);
\filldraw (4,0) node[below right] {$\mathbf{v}_3$} circle (2pt);
\filldraw (4,4) node[above right] {$\mathbf{v}_2$} circle (2pt);
\filldraw (0,4) node[above left] {$\mathbf{v}_1$} circle (2pt);
\draw (0,0) -- (0,4) -- (4,4) -- (4,0) -- (0,0);
\draw (0,0) -- (0,-2);
\draw (0,0) -- (-2,0);
\node[] at (-1.5,2) {$F_1$};
\node[] at (-1.5,-1.5) {$F_2$};
\node[] at (2,-1.5) {$F_3$};
\node[] at (2,2) {$F_4$};
    \end{tikzpicture}
    \end{center}
    \caption{Trinomial 5-minors \label{fig:trinomial-minors}}
  \end{subfigure}
  \begin{subfigure}[c]{.45\linewidth}
    \begin{center}
    \begin{tikzpicture}[scale=0.4]
\filldraw (0,0) node[below left] {$\mathbf{v}_4$} circle (2pt);
\filldraw (4,0) node[below right] {$\mathbf{v}_3$} circle (2pt);
\filldraw (4,4) node[above right] {$\mathbf{v}_2$} circle (2pt);
\filldraw (0,4) node[above left] {$\mathbf{v}_1$} circle (2pt);
\draw (0,0) -- (0,4) -- (4,4) -- (4,0) -- (0,0);
\node[] at (-0.5,2) {$F_1$};
\node[] at (2,-0.5) {$F_3$};
\node[] at (2,2) {$F_4$};
\node[] at (2,4.5) {$F_6$};
\node[] at (4.5,2) {$F_5$};
    \end{tikzpicture}
    \end{center}
    \caption{Binomial 5-minors \label{fig:binomial-minors}}
  \end{subfigure}
\caption{Common configurations that yield trinomial and binomial 5-minors \label{fig:3d-minors}}
\end{figure}
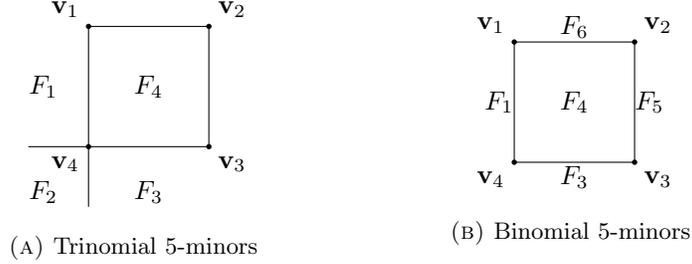
A quadrilateral facet alone (or dually, a vertex of degree four alone) yields a binomial 5-minor. We choose labels to be consistent with those of the previous proposition.

\begin{proposition} \label{prop:3d-binomials}
  Let $P$ be a 3-polytope and $F_4$ be a quadrilateral facet of $P$. Let $\vvec_1$, $\vvec_2$, $\vvec_3$, and $\vvec_4$ be the four vertices incident to $F_4$, and $F_1$, $F_3$, $F_5$, and $F_6$ be the facets adjacent to $F_4$, arranged as in Figure \ref{fig:3d-minors} (B). Then
  \[  h = u \left(x_{12}x_{31}x_{54}x_{63} - x_{13}x_{32}x_{51}x_{64}   \right) \]
  is a binomial 5-minor of $S_P(\x)$ for some variable $u$.

  Thus $x_{12}x_{31}x_{54}x_{63} - x_{13}x_{32}x_{51}x_{64}$ belongs to the slack ideal of $P$.
\end{proposition}

\begin{proof}
  The submatrix of $S_P(\x)$ with rows indexed by $F_1$, $F_3$, $F_5$, $F_6$, and $F_4$ (in that order) and columns indexed by $\vvec_1, \dots, \vvec_4$ and any other single vertex is of the form
  $ \begin{bmatrix}
    0 & x_{12} & x_{13} & 0 & \ast \\
    x_{31} & x_{32} & 0 & 0 & \ast \\
    x_{51} & 0 & 0 & x_{54} & \ast \\
    0 & 0 & 0 & 0 & y \\
    \ast & \ast & \ast & u & \ast
  \end{bmatrix} $
(where $\ast$ may denote either zero or a variable) and its determinant is exactly $h$. The binomial results from the fact that the upper left $4 \times 4$ matrix is just the slack matrix of a quadrilateral, as in Example \ref{ex:squareslack}.
\end{proof}

Propositions \ref{prop:3d-trinomials} and Proposition \ref{prop:3d-binomials} allow us to generate most of the trinomial $5$-minors in a quicker way than going through all possible minors (there are only few  sporadic trinomial 5-minors that are not obtained from the configuration of Figure \ref{fig:3d-minors} (A)). We can then use them with Proposition \ref{prop:incompatible} directly.
This 3-dimensional specific version of our approach was also implemented in SageMath, and made available\footnote{Functions ComplexPsdMinimality1$\_$3d and ComplexPsdMinimality2$\_$3d, for $L$ and $L'$ respectively, in \texttt{ComplexPsdMinimality.ipynb} available at \href{https://sites.google.com/view/jctorres}{https://sites.google.com/view/jctorres}}.


\subsection{$3$-polytopes with a vertex of degree six}

We begin by completely characterizing the case of 3-polytopes that either have a vertex of degree six or, by duality, some hexagonal facet.

\begin{proposition} \label{prop:deg6v}
  A complex psd-minimal 3-polytope with a vertex of degree six has exactly seven vertices.
\end{proposition}

\begin{proof}
  Suppose there is a complex psd-minimal 3-polytope $P$ with at least eight vertices including a vertex $\vvec$ of degree six. Let $\vvec_1, \dots, \vvec_6$ be the neighbors of $\vvec$ and $F_1, \dots, F_6$ be the facets incident to $\vvec$, labeled so that $F_i$ is incident to $\vvec_i$ and $\vvec_{i+1}$ (with indices taken modulo 6.)

  \textbf{Case 1:} Suppose that one of the six facets, say $F_1$, is not a triangle. Then there exists a vertex $\vvec'$ incident to $F_1$ but not to any of $F_2, \dots, F_6$. After appropriate row and column scalings, the submatrix of the slack matrix whose rows are indexed by $F_1, \dots F_6,F$ and columns by $\vvec_1, \dots, \vvec_6, \vvec', \vvec$, where $F$ is a facet not containing $\vvec$, is
  \[ \hat{M} = \begin{bmatrix}
    0 & 0 & 1 & x_{14} & x_{15} & 1 & 0 & 0\\
    1 & 0 & 0 & 1 & x_{25} & x_{26} & x_{27} & 0 \\
    x_{31} & 1 & 0 & 0 & 1 & x_{36} & x_{37} & 0\\
    1 & x_{42} & 1 & 0 & 0 & x_{46} & x_{47} & 0\\
    x_{51} & 1 & x_{53} & 1 & 0 & 0 & x_{57} & 0\\
    0 & x_{62} & x_{63} & x_{64} & 1 & 0 & x_{67} & 0\\
\ast & \ast & \ast & \ast & \ast & \ast & \ast & 1
  \end{bmatrix}. \]

  Let $M$ be the submatrix obtained from $\hat{M}$ by deleting the last row and column. Among the 4-minors of $M$ (and thus 5-minors of $\hat{M}$ and the scaled slack matrix) are the trinomials\footnote{See \texttt{AdditionalCalculations3.ipynb} in \href{https://sites.google.com/view/jctorres}{https://sites.google.com/view/jctorres}}
  \[ x_{14}-x_{31}x_{42}+1, \, x_{62}x_{27}x_{31} - x_{62}x_{37} + x_{67}, \, x_{42}x_{67} - x_{62}x_{47} + x_{62}x_{27} \]
whose standard forms are
  \[ x_{14}-x_{31}x_{42}+1, \, \frac{x_{62}x_{27}x_{31}}{x_{67}}-\frac{x_{62}x_{37}}{x_{67}} + 1, \, \frac{ x_{42}x_{67}}{x_{62}x_{27}} - \frac{x_{47}}{x_{27}} + 1.\]

    Since $\left( \frac{x_{62}x_{27}x_{31}}{x_{67}} \right) \left( \frac{ x_{42}x_{67}}{x_{62}x_{27}} \right) = x_{31}x_{42}$, $P$ cannot be complex psd-minimal by Proposition \ref{prop:incompatible}.

     \textbf{Case 2:} Suppose that all six of the facets incident to $\vvec$ are triangles. Then there exists a vertex $\vvec'$ incident to none of these six facets. Consider matrices $\hat{M}$ and $M$ analogously as in Case 1, but now there is a 1 in position $(F_1,\vvec')$ instead of a 0. For the scaled slack matrix $S_P$, by hypothesis, there exist $\boldsymbol{\alpha},\boldsymbol{\zeta}$ such that $S_P(\boldsymbol{\alpha})=\rank S_P(\boldsymbol{\zeta})\odot\overline{\rank S_P(\boldsymbol{\zeta})}$ with $\rank S_P(\boldsymbol{\alpha})=\rank S_P(\boldsymbol{\zeta})=4$. Now $\rank M(\boldsymbol{\alpha})\geq 3$ by Proposition \ref{prop:rank} since it contains the scaled symbolic slack matrix of a hexagon, and therefore $\rank\hat{M}(\boldsymbol{\alpha})\geq 4$. Thus $\rank\hat{M}(\boldsymbol{\alpha})=4$ and $\rank M(\boldsymbol{\alpha})=3$. Similarly, $\rank M(\boldsymbol{\zeta})=3$.
From the proof of Proposition \ref{prop:Pappus} it follows that such solutions will also satisfy $\alpha_{46} = \zeta_{46} = \alpha_{51} = \zeta_{51} = \alpha_{62} = \zeta_{62} = 1$, so we consider the modified matrix
    \[M' = \begin{bmatrix}
0 & 0 & 1 & x_{14} & x_{15} & 1 & 1\\
1 & 0 & 0 & 1 & x_{25} & x_{26} & x_{27}\\
x_{31} & 1 & 0 & 0 & 1 & x_{36} & x_{37}\\
1 & x_{42} & 1 & 0 & 0 & 1 & x_{47}\\
1 & 1 & x_{53} & 1 & 0 & 0 & x_{57}\\
0 & 1 & x_{63} & x_{64} & 1 & 0 & x_{67}
\end{bmatrix}.\]

Our strategy will be to parametrize (many of) the coordinates of $\boldsymbol{\alpha}$ and $\boldsymbol{\zeta}$ as rational functions of a small number of real variables $\cos \theta, \sin \theta, A$, where $\theta$ is a fixed angle. Then we will use Gr\"obner bases to show that such relations are not possible, which yields a contradiction.

Now the following polynomials belong to the ideal of 4-minors of $M'$\footnote{See \texttt{AdditionalCalculations3.ipynb} in \href{https://sites.google.com/view/jctorres}{https://sites.google.com/view/jctorres}}.

\[ \begin{array}{cc} x_{15}-1+x_{42} & x_{25} - x_{15}x_{26} + 1 \\
  x_{42}x_{53}-x_{15}x_{26} &  x_{37}-x_{31}x_{27}+x_{63}-x_{67} \\
  x_{25}x_{31} - x_{53} &
  x_{47}-x_{27}-x_{42}x_{67}+x_{26}-1
  \\
  x_{53}-x_{15}x_{63} & x_{57}+x_{26}-x_{27}-x_{67} \end{array}.\]

  Let $\zeta_{15} = a+bi$. From the trinomial $x_{15}-1+x_{42}$, we obtain that $\zeta_{42} = (1-a)-bi$. Then by Proposition \ref{prop:to} we have
  \[ 0 = \re \left( \zeta_{15}\overline{\zeta_{42}} \right) = a(1-a)-b^2.\]
  We conclude that $0 \leq a \leq 1$, so for some angle $\theta$ we can write $a = \cos^2 \theta$ and $b = \cos \theta \sin \theta$. Writing $c := \cos \theta$ and $s := \sin \theta$, we can now parametrize
  \[ \zeta_{15} = c^2 + i c s, \]
  \[ \zeta_{42} = s^2 - i c s.\]
Since $\zeta_{15}$ and $\zeta_{42}$ are assumed to be nonzero, $s$ and $c$ are also nonzero so we can and will divide by these parameters.

From the trinomial $x_{25} - x_{15}x_{26} + 1$ we immediately obtain from Proposition \ref{prop:to} that $\zeta_{25}$ is pure imaginary, so we parametrize
  \[ \zeta_{25} = iA\]
  and again $A$ must be nonzero.

  From applying the same trinomial directly to $\boldsymbol{\zeta}$ we obtain that $1 = \re(\zeta_{15}\zeta_{26}) =  c^2 \re(\zeta_{26})-  c s \im(\zeta_{26})$ and that $A = \im(\zeta_{15}\zeta_{26}) = c s \re(\zeta_{26})+  c^2 \im(\zeta_{26})$. We solve to obtain the parametrization
\[ \zeta_{26} = \left(1 + \frac{A s}{c} \right) + i \left( A - \frac{s}{c}  \right).\]

From the binomial $x_{42}x_{53}-x_{15}x_{26}$ we can directly obtain the parametrization
 \[ \zeta_{53}= \left( 1 - \frac{A c}{s} \right) + i \left( A + \frac{c}{s} \right) .\]
 Then from the binomials $x_{53}-x_{15}x_{63}$ and $x_{25}x_{31} - x_{53}$ we respectively obtain
 \[\zeta_{63}= \left( 2 - \frac{A c}{s} + \frac{A s}{c} \right) + i \left( 2A + \frac{c}{s} - \frac{s}{c} \right), \]
 \[ \zeta_{31} = \left( 1 + \frac{c}{A s}\right) + i \left(\frac{c}{s} - \frac{1}{A} \right).\]

 Using that $\alpha_{ij} = \abs{\zeta_{ij}}^2$ and simplifying via the relation $c^2 + s^2 = 1$, we obtain
 \[ \begin{array}{cccc} \alpha_{15} = c^2, & \alpha_{25} = A^2, & \alpha_{26} = \frac{1+A^2}{c^2}, & \alpha_{31} = \frac{1+A^2}{A^2 s^2}, \\
   \alpha_{42} = s^2, & \alpha_{53} = \frac{1+A^2}{s^2}, & \alpha_{63} = \frac{1+A^2}{c^2 s^2}.  &
   \end{array} \]

 We can now express $\zeta_{37}$, $\zeta_{47}$, and $\zeta_{57}$ in terms of the real parameters $c$, $s$ and $A$ as well as the remaining variables $\zeta_{27} := a_{27} + ib_{27}$ and $\zeta_{67} := a_{67} + ib_{67}$. To do this, we already have $x_{37} = x_{31}x_{27}-x_{63}+x_{67}$, $x_{57} = x_{27}+x_{67}-x_{26}$, and $x_{47} =x_{27}+x_{42}x_{67}-x_{26}+1$.
Similarly, we can express $\alpha_{37}$, $\alpha_{47}$, and $\alpha_{57}$ in terms of the same collection of variables.

 We now form an ideal in the ring $\CC[c,s,A,a_{27},b_{27},a_{67},b_{67}]$ from the relations $\alpha_{i7} = \abs{\zeta_{i7}}^2$ for $i \in \{ 3, 4, 5 \}$ as well the relation $c^2 + s^2 = 1$. From a Gr\"obner basis calculation in Sage\footnote{See \texttt{AdditionalCalculations4.ipynb} in \href{https://sites.google.com/view/jctorres}{https://sites.google.com/view/jctorres}}, we find that this ideal contains 1; that is, there are no solutions for the seven real parameters.
\end{proof}

There are four combinatorial 3-polytopes with exactly seven vertices including one of degree six. These are shown in Figure \ref{fig:v6}. All four turn out to have complex psd-minimal realizations. For each one we exhibit a complex matrix $M$ such that $S:=M \odot \overline{M}$ and  $M$ have rank 4, $S$ is a scaled slack matrix, and thus the pair $(S,M)$ witnesses its minimality\footnote{See \texttt{AdditionalCalculations5.ipynb} in \href{https://sites.google.com/view/jctorres}{https://sites.google.com/view/jctorres}}. In all four cases the first six rows, which index the facets incident to the degree six vertex, are as follows:
\[ \begin{bmatrix}
  0 & 0 & 0 & 1 & -i & \frac{1}{2}+\frac{1}{2}i & 1\\
0 & 1 & 0 & 0 & 1 & i & 2\\
0 & 2 & 1 & 0 & 0 & 1 & 2-2i\\
0 & 1 & \frac{1}{2}-\frac{1}{2}i & 1 & 0 & 0 & 1\\
0 & 1 & 1 & 2i & 1 & 0 & 0\\
0 & 0 & 1 & 2+2i & -2i & 1 & 0 \end{bmatrix}. \]
The remaining rows for each of the four polytopes are as follows:

\[ \begin{bmatrix}
  1 & 0 & 0 & 0 & 0 & 0 & 0
\end{bmatrix}, \:
\begin{bmatrix} 1 & 0 & -i & 2 & 0 & 0 & 0\\
  1 & 0 & 0 & 0 & 0 & 1 & -2i
\end{bmatrix},\]
\[ \begin{bmatrix}
1 & 0 & 0 & 1 & 0 & \frac{1}{2}+\frac{1}{2}i & 1\\
1 & -i & 0 & 0 & 0 & 1 & -2i\\
1 & -i & -i & 2 & 0 & 0 & 0\\
1 & 0 & \frac{1}{2} & 1+i & 0 & \frac{1}{2} & 0
\end{bmatrix}, \:
\begin{bmatrix}
1 & 0 & 0 & 0 & 2 & 2i & 2+2i\\
1 & -2i & -2i & 2-2i & 0 & 0 & 0\\
1 & 0 & 1 & 0 & 0 & 1 & 0
\end{bmatrix}.\]

The graphs of these polytopes, based on those available at Wolfram MathWorld \cite{Weisstein}, are shown in Figure \ref{fig:v6}. For their duals, see Chapter 5.1 of \cite{torres2020slack}.

\begin{figure}[ht]
\begin{center}
\begin{tikzpicture}[scale = 0.5]
\filldraw (0,0) node[above] {$\mathbf{v_1}$} circle (2pt);
\filldraw (0:2.5) node[right] {$\mathbf{v}_2$} circle (2pt);
\filldraw (60:2.5) node[right] {$\mathbf{v}_3$} circle (2pt);
\filldraw (120:2.5) node[left] {$\mathbf{v}_4$} circle (2pt);
\filldraw (180:2.5) node[left] {$\mathbf{v}_5$} circle (2pt);
\filldraw (240:2.5) node[left] {$\mathbf{v}_6$} circle (2pt);
\filldraw (300:2.5) node[right] {$\mathbf{v}_7$} circle (2pt);
\draw (0:2.5) -- (60:2.5) -- (120:2.5) -- (180:2.5) -- (240:2.5) -- (300:2.5) -- (0:2.5);
\draw (0,0) -- (0:2.5);
\draw (0,0) -- (60:2.5);
\draw (0,0) -- (120:2.5);
\draw (0,0) -- (180:2.5);
\draw (0,0) -- (240:2.5);
\draw (0,0) -- (300:2.5);
\end{tikzpicture}
\begin{tikzpicture}[scale = 0.5]
\filldraw (0,0) node[below] {$\mathbf{v}_2$} circle (2pt);
\filldraw (4,0) node[below] {$\mathbf{v}_3$} circle (2pt);
\filldraw (4,4) node[above] {$\mathbf{v}_4$} circle (2pt);
\filldraw (0,4) node[above] {$\mathbf{v}_5$} circle (2pt);
\draw (0,0) -- (0,4) -- (4,4) -- (4,0) -- (0,0);
\filldraw (3,2) node[right] {$\mathbf{v_1}$} circle (2pt);
\filldraw (1,3) node[left] {$\mathbf{v}_6$} circle (2pt);
\filldraw (1,1) node[left] {$\mathbf{v}_7$} circle (2pt);
\draw (0,0) -- (3,2) -- (1,1) -- (0,0);
\draw (0,4) -- (3,2) -- (1,3) -- (0,4);
\draw (1,1) -- (1,3);
\draw (4,0) -- (3,2) -- (4,4);
\end{tikzpicture}
\begin{tikzpicture}[scale = 0.5]
\filldraw (0,0) node[below] {$\mathbf{v}_2$} circle (2pt);
\filldraw (5,0) node[below] {$\mathbf{v}_5$} circle (2pt);
\filldraw (60:5) node[above] {$\mathbf{v}_7$} circle (2pt);
\draw (0,0) -- (5,0) -- (60:5) -- (0,0);
\filldraw (2,1) node[below] {$\mathbf{v}_3$} circle (2pt);
\filldraw (2,2) node[left] {$\mathbf{v_1}$} circle (2pt);
\filldraw (2.5,1) node[right] {$\mathbf{v}_4$} circle (2pt);
\filldraw (3,2) node[right] {$\mathbf{v}_6$} circle (2pt);
\draw (0,0) -- (2,1) -- (2,2) -- (0,0);
\draw (2,1) -- (2.5,1) -- (5,0) -- (2,1);
\draw (2.5,1) -- (2,2) -- (5,0) -- (3,2) -- (2,2) -- (60:5) -- (3,2);
\end{tikzpicture}
\begin{tikzpicture}[scale = 0.5]
\filldraw (0,0) node[below] {$\mathbf{v}_2$} circle (2pt);
\filldraw (4,0) node[below] {$\mathbf{v}_4$} circle (2pt);
\filldraw (4,4) node[above] {$\mathbf{v}_5$} circle (2pt);
\filldraw (0,4) node[above] {$\mathbf{v}_7$} circle (2pt);
\draw (0,0) -- (0,4) -- (4,4) -- (4,0) -- (0,0);
\filldraw (2,1) node[below] {$\mathbf{v}_3$} circle (2pt);
\filldraw (2,2) node[right] {$\mathbf{v_1}$} circle (2pt);
\filldraw (2,3) node[above] {$\mathbf{v}_6$} circle (2pt);
\draw (0,0) -- (2,1) -- (4,0);
\draw (2,1) -- (2,2) -- (2,3);
\draw (0,4) -- (2,3) -- (4,4);
\draw (0,0) -- (2,2) -- (4,4);
\draw (0,4) -- (2,2) -- (4,0);
\end{tikzpicture}
\end{center}
\caption{Polyhedral graphs with seven vertices including one of degree six and a complex psd-minimal realization\label{fig:v6}}
  \end{figure}
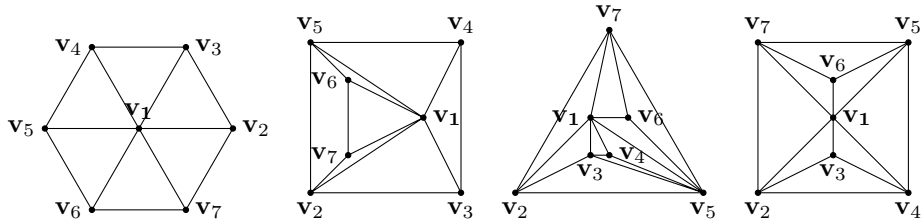

\subsection{Doubly 3/4 polytopes}
 From Proposition \ref{prop:346} and \ref{prop:deg6v}, it follows that the only 3-polytopes that are candidates for complex psd-minimality, other than the four shown in Figure \ref{fig:v6}, are those that satisfy the condition

 (*) $P$ is a 3-polytope whose facets are all triangles or quadrilaterals and whose vertices are all of degree three or four,

which we will call \emph{doubly 3/4 polytopes}. We summarize how many such polytopes exist up to 13 vertices and what we know about complex psd-minimality.

\textbf{4 and 5 vertices}: Any 3-polytope with 4 or 5 vertices is doubly 3/4. All of these (that is, all tetrahedra, quadrilateral bipyramids, and bisimplices) are psd-minimal, hence also complex psd-minimal, by the general result for $d$-polytopes with at most $d+2$ vertices.

\textbf{6 vertices}: There are seven combinatorial 3-polytopes with six vertices, of which four are doubly 3/4. Their graphs are shown in Figure \ref{fig:hex} (see Chapter 5.2 of \cite{torres2020slack} for their duals). All four possess complex psd-minimal realizations for the following reasons. The first is the octahedron and the second is the triangular prism, which are known to be real psd-minimal. For the others, it will suffice to exhibit a scaled slack matrix $S$ and a complex matrix $M$, both of rank four such that $S=M\odot\overline{M}$. Such matrices $M$\footnote{See \texttt{AdditionalCalculations5.ipynb} in \href{https://sites.google.com/view/jctorres}{https://sites.google.com/view/jctorres}} for the third and fourth of these polytopes are given respectively by

\[ M_3=\begin{bmatrix}
0 & 0 & 1 & 1 & 0 & 1\\
1 & 0 & 1 & 1-i & 0 & 0\\
i & 0 & 0 & 1 & 1 & 0\\
i & 1 & 0 & 0 & 1 & 0\\
1 & 1-i & 1 & 0 & 0 & 0\\
0 & 1 & 1 & 0 & 0 & 1\\
0 & 0 & 0 & 0 & -i & 1
\end{bmatrix}, \; M_4 = \begin{bmatrix}
0 & 1 & 1 & 0 & 0 & 0\\
0 & 0 & 1 & 1 & 1 & 0\\
1 & 1 & 0 & 0 & 0 & 1\\
0 & 0 & 0 & 0 & \frac{1}{2}+\frac{1}{2}i & 1\\
1 & 0 & 0 & 1 & \frac{1}{2}-\frac{1}{2}i & 0\\
1+i & 1 & 0 & i & 0 & 0
\end{bmatrix}.\]

\begin{figure}[ht]
\begin{center}
\begin{tikzpicture}[scale = 0.5]
\filldraw (90:1.5) node[right] {$\mathbf{v}_4$} circle (2pt);
\filldraw (210:1.5) node[left] {$\mathbf{v}_5$} circle (2pt);
\filldraw (330:1.5) node[right] {$\mathbf{v}_6$} circle (2pt);
\filldraw (90:3) node[above] {$\mathbf{v}_3$} circle (2pt);
\filldraw (210:3) node[below] {$\mathbf{v}_1$} circle (2pt);
\filldraw (330:3) node[below] {$\mathbf{v}_2$} circle (2pt);
\draw (210:3) -- (330:3) -- (90:3) -- (210:3);
\draw (210:1.5) -- (330:1.5) -- (90:1.5) -- (210:1.5);
\draw (90:3) -- (90:1.5);
\draw (210:3) -- (210:1.5);
\draw (330:3) -- (330:1.5);
\end{tikzpicture}
\begin{tikzpicture}[scale = 0.5]
\filldraw (30:1) node[below] {$\mathbf{v}_6$} circle (2pt);
\filldraw (150:1) node[below] {$\mathbf{v}_5$} circle (2pt);
\filldraw (270:1) node[below] {$\mathbf{v}_2$} circle (2pt);
\filldraw (90:3) node[above] {$\mathbf{v}_4$} circle (2pt);
\filldraw (210:3) node[below] {$\mathbf{v}_1$} circle (2pt);
\filldraw (330:3) node[below] {$\mathbf{v}_3$} circle (2pt);
\draw (210:3) -- (330:3) -- (90:3) -- (210:3);
\draw (210:3) -- (150:1) -- (270:1) -- (210:3);
\draw (330:3) -- (270:1) -- (30:1) -- (330:3);
\draw (90:3) -- (150:1) -- (30:1) -- (90:3);
\end{tikzpicture}
\begin{tikzpicture}[scale = 0.5]
\filldraw (30:1) node[right] {$\mathbf{v}_6$} circle (2pt);
\filldraw (150:1) node[left] {$\mathbf{v}_5$} circle (2pt);
\filldraw (270:1.5) node[below] {$\mathbf{v}_2$} circle (2pt);
\filldraw (90:3) node[above] {$\mathbf{v}_4$} circle (2pt);
\filldraw (210:3) node[below] {$\mathbf{v}_1$} circle (2pt);
\filldraw (330:3) node[below] {$\mathbf{v}_3$} circle (2pt);
\draw (210:3) -- (330:3) -- (90:3) -- (210:3);
\draw (210:3) -- (150:1) -- (270:1.5);
\draw (330:3) -- (30:1) -- (270:1.5);
\draw (90:3) -- (150:1) -- (30:1) -- (90:3);
\end{tikzpicture}
\begin{tikzpicture}[scale = 0.5]
\filldraw (0,0) node[below] {$\mathbf{v}_1$} circle (2pt);
\filldraw (0,4) node[above] {$\mathbf{v}_2$} circle (2pt);
\filldraw (4,4) node[above] {$\mathbf{v}_3$} circle (2pt);
\filldraw (4,0) node[below] {$\mathbf{v}_4$} circle (2pt);
\draw (0,0) -- (0,4) -- (4,4) -- (4,0) -- (0,0);
\filldraw (3,2) node[right] {$\mathbf{v}_5$} circle (2pt);
\filldraw (1,2) node[left] {$\mathbf{v}_6$} circle (2pt);
\draw (4,0) -- (3,2) -- (1,2) -- (0,0);
\draw (0,4) -- (1,2) -- (4,4) -- (3,2);
\end{tikzpicture}
\end{center}
\caption{Polyhedral graphs with six vertices and a complex psd-minimal realization\label{fig:hex}}
\end{figure}
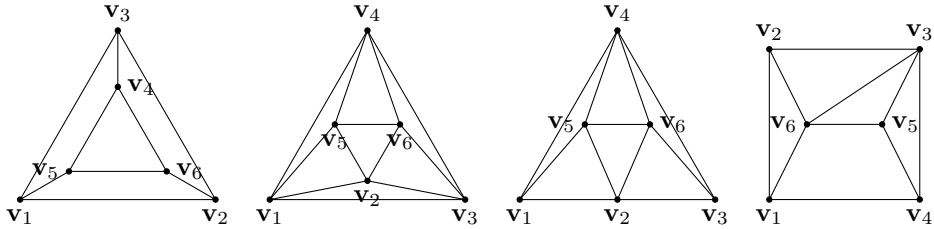

\textbf{7 vertices:} There are 34 combinatorial 3-polytopes with seven vertices, of which seven are doubly 3/4. One of these, which is self-dual and is shown in Figure~\ref{fig:7obstructed}, can be ruled out using trinomials. Specifically, a scaled slack matrix is
\[S=\begin{bmatrix}
0 & 0 & 0 & 1 & 1 & x_{16} & 1\\
0 & 1 & 0 & 0 & 1 & x_{26} & x_{27}\\
0 & 1 & 1 & 0 & 0 & 0 & x_{37}\\
0 & 0 & 1 & 1 & 0 & 1 & 0\\
1 & 1 & 0 & 0 & x_{55} & 0 & x_{57}\\
1 & 0 & 0 & x_{64} & x_{65} & 0 & 0\\
1 & x_{72} & x_{73} & x_{74} & 0 & 0 & 0
\end{bmatrix}\]
where we use that rows 1 through 4 and columns 2 through 5 represent the slack matrix of a quadrilateral to force an extra 1 in position $(1,5)$ as in Example \ref{ex:squareslack}.

\begin{figure}[ht]
  \begin{center}
    \begin{tikzpicture}[scale=0.5]
\filldraw (0,0) node[below] {$\mathbf{v}_1$} circle (2pt);
\filldraw (4,0) node[below] {$\mathbf{v}_2$} circle (2pt);
\filldraw (4,4) node[above] {$\mathbf{v}_7$} circle (2pt);
\filldraw (0,4) node[above] {$\mathbf{v}_5$} circle (2pt);
\draw (0,0) -- (0,4) -- (4,4) -- (4,0) -- (0,0);
\filldraw (1,1) node[below] {$\mathbf{v}_4$} circle (2pt);
\filldraw (3,1) node[below] {$\mathbf{v}_3$} circle (2pt);
\filldraw (2,3) node[above] {$\mathbf{v}_6$} circle (2pt);
\draw (0,0) -- (3,1) -- (2,3) -- (1,1);
\draw (0,0) -- (1,1) -- (3,1) -- (4,0) -- (3,1);
\draw (0,4) -- (2,3) -- (4,4);
    \end{tikzpicture}
  \end{center}
  \caption{A polyhedral graph with seven vertices with no complex psd-minimal realizations\label{fig:7obstructed}}
  \end{figure}
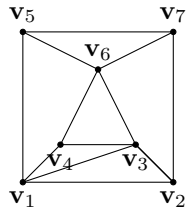

We can check computationally that the ideal of 5-minors of $S$ contains the trinomials\footnote{See \texttt{AdditionalCalculations6.ipynb} in \href{https://sites.google.com/view/jctorres}{https://sites.google.com/view/jctorres}}
\[ x_{37}-x_{27}+1, \; x_{16}x_{37}-x_{16}x_{57}+1, \;  x_{26}-x_{16}+1.\]
Since $\frac{x_{16}x_{37}}{x_{37}} = x_{16}$, this combinatorial polytope has no complex psd-minimal realizations by Proposition \ref{prop:incompatible}.

The remaining six combinatorial polytopes are shown in Figure \ref{fig-hep} (see Chapter 5.3 of \cite{torres2020slack} for their duals) and have all complex psd-minimal realizations\footnote{See \texttt{AdditionalCalculations5.ipynb} in \href{https://sites.google.com/view/jctorres}{https://sites.google.com/view/jctorres}}.
\begin{figure}[h]
\begin{subfigure}{1\textwidth}
\centering
\begin{tikzpicture}[scale = 0.6]
\filldraw (0,0) node[below] {$\mathbf{v}_1$} circle (2pt);
\filldraw (4,0) node[below] {$\mathbf{v}_2$} circle (2pt);
\filldraw (4,4) node[above] {$\mathbf{v}_3$} circle (2pt);
\filldraw (0,4) node[above] {$\mathbf{v}_4$} circle (2pt);
\draw (0,0) -- (0,4) -- (4,4) -- (4,0) -- (0,0);
\filldraw (3,2) node[right] {$\mathbf{v}_7$} circle (2pt);
\filldraw (1,3) node[above] {$\mathbf{v}_5$} circle (2pt);
\filldraw (1,1) node[below] {$\mathbf{v}_6$} circle (2pt);
\draw (0,0) -- (1,1);
\draw (1,1) -- (3,2) -- (1,3) -- (0,4);
\draw (1,1) -- (1,3);
\draw (4,0) -- (3,2) -- (4,4);
\end{tikzpicture}
\begin{tikzpicture}[scale = 0.6]
\filldraw (0,0) node[below] {$\mathbf{v}_1$} circle (2pt);
\filldraw (4,0) node[below] {$\mathbf{v}_2$}  circle (2pt);
\filldraw (4,4) node[above] {$\mathbf{v}_7$} circle (2pt);
\filldraw (0,4) node[above] {$\mathbf{v}_5$} circle (2pt);
\draw (0,0) -- (0,4) -- (4,4) -- (4,0) -- (0,0);
\filldraw (1,2) node[left] {$\mathbf{v}_4$} circle (2pt);
\filldraw (2,2) node[below] {$\mathbf{v}_3$} circle (2pt);
\filldraw (3,2) node[right] {$\mathbf{v}_6$} circle (2pt);
\draw (1,2) -- (2,2) -- (3,2) -- (4,4) -- (0,0) -- (1,2) -- (0,4);
\draw (3,2) -- (4,0);
\end{tikzpicture}
\begin{tikzpicture}[scale = 0.6]
\filldraw (0,0) node[below] {$\mathbf{v}_5$} circle (2pt);
\filldraw (4,0) node[below] {$\mathbf{v}_6$} circle (2pt);
\filldraw (4,4) node[above] {$\mathbf{v}_7$} circle (2pt);
\filldraw (0,4) node[above] {$\mathbf{v}_4$} circle (2pt);
\filldraw (2,2.5) node[above] {$\mathbf{v}_3$} circle (2pt);
\draw (0,0) -- (0,4) -- (4,4) -- (4,0) -- (0,0);
\filldraw (1,1) node[below] {$\mathbf{v}_1$} circle (2pt);
\filldraw (3,1) node[below] {$\mathbf{v}_2$} circle (2pt);
\draw (1,1) -- (3,1) -- (2,2.5) -- (0,4) -- (1,1) -- (2,2.5) -- (4,4);
\draw (0,0) -- (1,1);
\draw (4,0) -- (3,1);
\draw (3,1) -- (4,4);
\end{tikzpicture}
\begin{tikzpicture}[scale = 0.6]
\filldraw (0,0) node[below] {$\mathbf{v}_2$} circle (2pt);
\filldraw (4,0) node[below] {$\mathbf{v}_7$} circle (2pt);
\filldraw (4,4) node[above] {$\mathbf{v}_5$} circle (2pt);
\filldraw (0,4) node[above] {$\mathbf{v}_1$} circle (2pt);
\draw (0,0) -- (0,4) -- (4,4) -- (4,0) -- (0,0);
\filldraw (1,1) node[left] {$\mathbf{v}_3$} circle (2pt);
\filldraw (3,3) node[above] {$\mathbf{v}_4$} circle (2pt);
\filldraw (3,1) node[right] {$\mathbf{v}_6$} circle (2pt);
\draw (1,1) -- (3,1) -- (3,3) -- (0,4) -- (1,1) -- (4,0);
\draw (0,0) -- (1,1);
\draw (4,0) -- (3,1);
\draw (4,4) -- (3,3);
\draw (3,1) -- (4,4);
\end{tikzpicture}
\end{subfigure}
\begin{subfigure}{1\textwidth}
\centering
\begin{tikzpicture}[scale = 0.6]
\filldraw (0,0) node[below] {$\mathbf{v}_7$} circle (2pt);
\filldraw (90:1.5) node[right] {$\mathbf{v}_6$} circle (2pt);
\filldraw (210:1.5) node[below] {$\mathbf{v}_4$} circle (2pt);
\filldraw (330:1.5) node[below] {$\mathbf{v}_5$} circle (2pt);
\filldraw (90:3) node[above] {$\mathbf{v}_3$} circle (2pt);
\filldraw (210:3) node[below] {$\mathbf{v}_1$} circle (2pt);
\filldraw (330:3) node[below] {$\mathbf{v}_2$} circle (2pt);
\draw (210:3) -- (330:3) -- (90:3) -- (210:3);
\draw (210:1.5) -- (330:1.5) -- (90:1.5) -- (210:1.5);
\draw (210:3) -- (0,0) -- (330:3);
\draw (0,0) -- (90:3);
\end{tikzpicture}
\begin{tikzpicture}[scale = 0.6]
\filldraw (0,0) node[above] {$\mathbf{v}_7$} circle (2pt);
\filldraw (90:3) node[above] {$\mathbf{v}_3$} circle (2pt);
\filldraw (210:3) node[below] {$\mathbf{v}_1$} circle (2pt);
\filldraw (330:3) node[below] {$\mathbf{v}_2$} circle (2pt);
\filldraw (30:0.75) node {} circle(2pt);
\filldraw (150:0.75) node {} circle(2pt);
\filldraw (270:0.75) node {} circle(2pt);
\filldraw (30:1.1) node {$\mathbf{v}_6$};
\filldraw (150:1.1) node {$\mathbf{v}_4$};
\filldraw (270:1.1) node {$\mathbf{v}_5$};
\draw (210:3) -- (330:3) -- (90:3) -- (210:3);
\draw (210:3) -- (270:0.75) -- (0,0) -- (150:0.75) -- (210:3);
\draw (90:3) -- (150:0.75);
\draw (90:3) -- (30:0.75) -- (0,0);
\draw (270:0.75) -- (330:3) -- (30:0.75);
\end{tikzpicture}
\end{subfigure}
\caption{Polyhedral graphs with seven vertices, none of degree 6, and a complex psd-minimal realization\label{fig-hep}}
\end{figure}
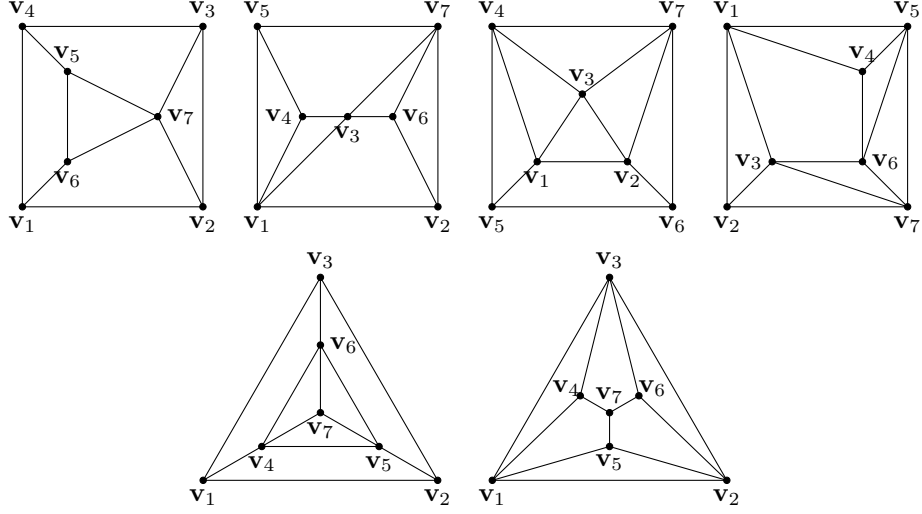

\textbf{8 or more vertices} We use the software package Plantri \cite{Plantri} to enumerate the graphs of doubly 3/4 polytopes with 8, 9, 10, 11, 12, and 13 vertices\footnote{Lists of these polytopes are available in \href{https://sites.google.com/view/jctorres}{https://sites.google.com/view/jctorres}}. We compare with the number for which complex psd-minimality is ruled out by our obstruction, using the trinomials and binomials from Proposition \ref{prop:3d-trinomials} and Proposition \ref{prop:3d-binomials}, and with the total number of 3-polytopes with the corresponding number of vertices (see sequence A000944 in the On-Line Encyclopedia of Integer Sequences \cite{Sloane}) in Table \ref{table:count-polytopes}.

\begin{table}[h]
\centering
\begin{tabular}{l l | l | l | l | l | l | l}
\# of vertices & 8 & 9 & 10 & 11 & 12 & 13 \\
\hline
\# of 3-polytopes & 257 & 2606 & 32300 & 440564 & 6384634 & 96262938 \\
\# of doubly 3/4 polytopes & 18 & 28 & 56 & 78 & 138 & 196  \\
\# ruled out & 2 & 4 & 6 & 6 & 7  & 9
\end{tabular}
\caption{Number of doubly 3/4 3-polytopes and those for which complex psd-minimality is ruled out by binomials and trinomials}
\label{table:count-polytopes}
\end{table}

\section{Conclusions and Questions}
We completely characterize complex psd-minimal polygons, including precisely which realizations of the hexagon, and rule out complex psd-minimality for the overwhelming majority of 3-polytopes by the limits on vertex and facet degrees. However, our techniques are not sufficient for a complete characterization in dimension three.

Both the complex psd-minimality of Pappus hexagons and the various examples in dimension three suggest that there are many more complex than real psd-minimal polytopes. In particular, it has been conjectured \cite[Conjecture 23]{bohn2019enumeration} that for every 2-level polytope $P$, the number of vertices $f_0(P)$ and the number of facets $f_{d-1}(P)$ satisfy
\[ f_0(P)f_{d-1}(P) \leq d2^{d+1}.\]
This conjecture is open not only for 2-level polytopes but also for the larger class of real psd-minimal polytopes. However our example of the hexagon shows it cannot extend to the complex psd-minimal case, since the left side of the inequality would be 36 and the right side 16.

It is known that there are only finitely many (combinatorial) complex psd-minimal polytopes in each dimension; this follows from \cite[Corollary 4.18]{GPT13} along with the fact that the complex psd-rank is at most twice the real psd-rank. However, the bound on the number of vertices of a complex psd-minimal 3-polytope obtained from this result would be huge.

\begin{problem}
What is the maximum number of vertices (or of vertices times facets) of a complex psd-minimal $d$-polytope for each $d \geq 3$?
\end{problem}

Although it follows from our characterization of trinomial minors that a doubly 3/4 polytope with many vertices and facets also has many such minors, the experiments suggest that these trinomials are often not enough to generate local obstructions to complex psd-minimality. However, we have seen that there are often many trinomials in the ideal of 5-minors that are not themselves minors.
\begin{problem}
Characterize classes of trinomials in the slack ideal which are not minors.
\end{problem}
Such trinomials could play an important role not only in our complex psd-minimality calculations but also in providing non-realizability  certificates. Finally, it is likely that local obstructions via trinomials will prove to be insufficient for a complete characterization of complex psd-minimality, so other ideas will need to be developed.
\begin{problem}
Is there a more global type of systematic obstruction to complex psd-minimality?
\end{problem}

\bibliographystyle{alpha}
\bibliography{trinomials}
\end{document}